\documentclass[11pt]{article}

\usepackage[a4paper,margin=2.5cm]{geometry}
\usepackage[utf8]{inputenc}
\usepackage[T1]{fontenc}
\usepackage{amsmath}
\usepackage{amssymb}
\usepackage{amsfonts}
\usepackage{amsthm}
\usepackage{mathtools}
\usepackage{graphicx}
\usepackage{booktabs}
\usepackage{array}
\usepackage{longtable}
\usepackage{multirow}
\usepackage{hyperref}
\usepackage{enumitem}
\usepackage{setspace}
\usepackage{cite}

\hypersetup{
	colorlinks=true,
	linkcolor=blue,
	citecolor=blue,
	urlcolor=blue
}

\newtheorem{definition}{Definition}[section]
\newtheorem{theorem}[definition]{Theorem}

\newtheorem{proposition}[definition]{Proposition}
\newtheorem{remark}[definition]{Remark}
\newtheorem{corollary}[definition]{Corollary}

\title{\bf Mechanical Implications of the Teleportation of
	Rigid Extended Bodies from the Perspective of Classical
	Mechanics and Mathematics}

\author{
	Roblêdo Mak's Miranda Sette\thanks{%
		Universidade Federal da Grande Dourados (UFGD), Faculdade de
		Ciências Exatas e Tecnologia (FACET), Dourados, MS, Brazil.
		E-mail: \texttt{robledosette@ufgd.edu.br}. ORCID: \texttt{0000-0003-2664-8748}.}
}

\date{\today}

\begin{document}
	
	\maketitle
	
	\begin{abstract}
		We develop a rigorous framework for the kinematics of a material
		point whose trajectory contains a jump discontinuity, modeled as a
		smooth curve on an interval punctured at one instant, with finite
		one-sided limits of position and velocity. A position jump is
		called an instantaneous teleportation. The event is treated as
		exogenous to Newtonian dynamics, imposed only before and after the
		removed instant. Separating the position and velocity
		discontinuities, we prove that the net linear impulse equals the
		momentum jump $\mathbf J=m\,\Delta\mathbf v$, carried entirely by
		the velocity discontinuity, while the position jump contributes
		none. This yields a \emph{momentum-compatibility} condition: an
		instantaneous displacement transfers no net impulse if and only if
		the velocity is preserved. For rigid bodies it splits into
		preservation of translational and angular velocity. For the
		rotating, revolving Earth, a kinematic fact governs the estimates:
		for two points of the same Earth at the same instant the orbital
		velocity cancels, bounding the jump by $2\omega R_{\oplus}\approx
		9.3\times10^{2}\ \mathrm{m\,s^{-1}}$; the larger scale
		$6.0\times10^{4}\ \mathrm{m\,s^{-1}}$ arises only between orbital
		epochs or frames. The internal stress depends on how momentum is
		delivered, and the implied strains lie outside linear elasticity,
		so the figures are indicative.
	\end{abstract}
	
	
	\vspace{0.6cm}
	
	\begin{center}
		\textbf{\large Resumo}
	\end{center}
	
	\noindent
	Desenvolvemos um arcabouço rigoroso para a cinemática de um ponto
	material cuja trajetória contém uma descontinuidade de salto,
	modelada como curva suave em um intervalo do qual se remove um
	instante, com limites laterais finitos de posição e velocidade. Um
	salto de posição é chamado de teletransporte instantâneo. O evento é
	tratado como externo à dinâmica newtoniana, imposta apenas antes e
	depois do instante removido. Separando as descontinuidades de posição
	e de velocidade, demonstramos que o impulso linear líquido é igual ao
	salto de momento $\mathbf J=m\,\Delta\mathbf v$, carregado
	inteiramente pela descontinuidade de velocidade, enquanto o salto de
	posição não contribui. Disso resulta uma condição de
	\emph{compatibilidade de momento}: um deslocamento instantâneo não
	transfere impulso líquido se, e somente se, a velocidade for
	preservada. Para corpos rígidos, isso se desdobra na preservação das
	velocidades translacional e angular. Para a Terra em rotação e
	revolução, um fato cinemático governa as estimativas: para dois
	pontos da mesma Terra no mesmo instante a velocidade orbital se
	cancela, limitando o salto a $2\omega R_{\oplus}\approx
	9{,}3\times10^{2}\ \mathrm{m\,s^{-1}}$; a escala maior
	$6{,}0\times10^{4}\ \mathrm{m\,s^{-1}}$ surge apenas entre épocas
	orbitais ou referenciais. A tensão interna depende de como o momento
	é entregue, e as deformações implícitas ficam fora da elasticidade
	linear, de modo que os números são indicativos.
	
	\vspace{0.2cm}
	
	\noindent
	\textbf{Palavras-chave:}
	Descontinuidade de salto;
	Limites laterais;
	Mecânica clássica;
	Momento linear;
	Impulso;
	Funções de variação limitada;
	Elasticidade;
	Mecânica do contínuo;
	Análise de tensões;
	Referenciais.
	
	\section{Introduction}
	
	Consider a material point moving along a smooth trajectory in an
	inertial reference frame, and suppose that at a single instant its
	position jumps discontinuously from one point of space to another,
	with no intermediate motion. Such a jump discontinuity of a
	vector-valued function of time is an elementary object of classical
	analysis: the one-sided limits exist and are finite, yet no value is
	assigned at the instant of the jump. The purpose of this paper is to
	study the kinematics of trajectories of this kind and to determine,
	within Newtonian mechanics, the mechanical consequences that such a
	discontinuity necessarily entails.
	
	The motivating physical picture is instantaneous teleportation.
	While teleportation has occupied a prominent place in scientific
	speculation, the only physically established notion currently known
	is quantum teleportation, which transfers quantum information rather
	than matter itself
	\cite{Bennett1993,Bouwmeester1997,NielsenChuang2010}; the purely
	mechanical consequences of an instantaneous relocation of a
	macroscopic body have received little attention. We do not assume any
	mechanism capable of producing such a relocation. Instead, we take
	the discontinuous trajectory as a mathematical primitive and ask what
	classical mechanics forces upon it, independently of any physical
	implementation.
	
	The analytic starting point, made precise in
	Section~\ref{sec:model}, is a curve of class $C^2$ defined on an open
	interval punctured at one instant $t_0$, whose one-sided limits of
	position and velocity exist and are finite. Removing $t_0$ from the
	domain expresses the requirement that the point be absent from space
	at the instant of the jump. A modeling decision made throughout, and
	stated explicitly in Remark~\ref{rem:exogenous}, is that the
	teleportation event is \emph{exogenous} to Newtonian dynamics: the
	equation of motion is imposed only on the two open sub-intervals,
	never through $t_0$. This is not a technical convenience but a
	necessity, because a discontinuous trajectory inserted into Newton's
	law in the sense of distributions produces, from the position jump, a
	term proportional to the derivative of the Dirac distribution---an
	object more singular than an ordinary impulse. By keeping the event
	outside the dynamics we compare only the one-sided states and read
	off the net impulse the mechanism must supply.
	
	The central observation is that two distinct discontinuities must be
	separated: the jump of \emph{position}, which is the teleportation
	itself, and the jump of \emph{velocity}, which need not accompany it.
	Our main result (Theorem~\ref{thm:impulse}) shows that the net linear
	impulse required by the event equals the momentum jump
	$\mathbf J=m\,\Delta\mathbf v$, carried entirely by the velocity
	discontinuity; the position jump alone transfers no net linear
	impulse. This leads to a \emph{momentum-compatibility} condition
	(Corollary~\ref{cor:compat}): an instantaneous displacement is free of
	net impulsive momentum transfer precisely when the velocity vector is
	preserved through it. We are careful (Remark~\ref{rem:not-admissible})
	not to overstate this: momentum compatibility is necessary, but not
	sufficient, for the absence of internal stress.
	
	The remainder of the paper develops the consequences of this
	principle. Section~\ref{sec:background} fixes the mechanical
	framework. Section~\ref{sec:model} introduces the model and proves
	the impulse theorem in a measure-theoretic form. Section~\ref{sec:curve}
	instantiates the trajectory for a point on the rotating and revolving
	Earth, deriving the kinematic bound for two points of the same Earth
	at the same instant and distinguishing it from the inter-epoch and
	inter-frame cases. Section~\ref{sec:extended} generalizes the theory
	to extended rigid bodies, carrying the orientation explicitly, and
	separates net impulse, body force and internal stress.
	Section~\ref{sec:deformation} recalls the elements of linear
	elasticity, translates the induced impulse into an order-of-magnitude
	stress, and states plainly the range of validity of the estimate.
	Section~\ref{sec:conclusion} concludes.
	
	\section{Physical Background}
	\label{sec:background}
	
	The purpose of this section is to establish the mechanical framework
	used throughout the paper. Since the analysis concerns the
	instantaneous relocation of macroscopic bodies, only classical
	mechanics and continuum mechanics are employed. No assumptions
	concerning quantum gravity, relativistic field theory or microscopic
	teleportation mechanisms are required \cite{Goldstein2002,LandauLifshitz1976}.
	The objective is to derive necessary mechanical conditions that any
	hypothetical instantaneous teleportation process must satisfy,
	independently of its physical implementation.
	
	\subsection{Classical Mechanics}
	
	Throughout this paper we assume that every material point moves in a
	reference frame
	\[
	\mathcal I=(O,\mathbf e_1,\mathbf e_2,\mathbf e_3)
	\]
	whose origin $O$ is placed at the Solar System barycenter and which is
	taken to be inertial to the accuracy required here. Away from the
	teleportation event, the trajectory of a material point is represented
	by a twice continuously differentiable curve
	\[
	\gamma:I\subset\mathbb R\longrightarrow\mathbb R^3,
	\qquad t\mapsto\gamma(t),
	\]
	where $\gamma(t)$ denotes its spatial position at the instant $t$. The
	instantaneous velocity is the tangent vector
	$\mathbf v(t)=\dot{\gamma}(t)$, and the acceleration is
	$\mathbf a(t)=\ddot{\gamma}(t)$. The assumption that the trajectory is
	of class $C^2$ reflects the fact that ordinary mechanical motion does
	not exhibit discontinuous changes of velocity.
	
	\subsection{Linear Momentum}
	
	For a particle of constant mass $m$, the linear momentum is
	$\mathbf p(t)=m\,\mathbf v(t)$. Whenever the mass remains constant,
	\[
	\frac{d\mathbf p}{dt}=m\,\mathbf a,
	\]
	and Newton's Second Law may be written $\mathbf F=d\mathbf p/dt$, which
	expresses the fact that every net external force produces a variation
	of linear momentum. This formulation is valid in any inertial frame.
	
	\subsection{Impulse}
	
	If a force acts during the finite interval $[t_1,t_2]$, the mechanical
	impulse is
	\[
	\mathbf J=\int_{t_1}^{t_2}\mathbf F(t)\,dt.
	\]
	Combining this with Newton's Second Law yields the impulse--momentum
	theorem,
	\[
	\mathbf J=\Delta\mathbf p=\mathbf p(t_2)-\mathbf p(t_1),
	\]
	so every variation of momentum corresponds to a nonzero impulse.
	
	\subsection{Impulsive Interactions}
	\label{subsec:impulsive}
	
	Ordinary mechanical systems evolve continuously in time, so both
	position and velocity remain continuous whenever the acting forces are
	finite. Suppose, on the contrary, that the velocity undergoes a
	discontinuous change $\Delta\mathbf v=\mathbf v_2-\mathbf v_1\neq
	\mathbf 0$ across a vanishingly small interval $\Delta t$. The
	corresponding momentum change is $\Delta\mathbf p=m\,\Delta\mathbf v$,
	and the average force required to produce it is
	\[
	\mathbf F_{\mathrm{avg}}=\frac{\Delta\mathbf p}{\Delta t}
	=\frac{m\,\Delta\mathbf v}{\Delta t}.
	\]
	In the idealized limit $\Delta t\to 0$ with $\Delta\mathbf v$ fixed and
	nonzero, the magnitude $\|\mathbf F_{\mathrm{avg}}\|$ diverges (the
	direction being fixed), so the average force is unbounded while its
	time integral remains equal to the finite impulse
	$\mathbf J=m\,\Delta\mathbf v$. This is the standard idealization of an
	impulsive force, made precise below by treating the momentum as a
	function of bounded variation whose distributional derivative carries
	an atom at the instant of the jump.
	
	\section{A Mathematical Model for Teleportation}
	\label{sec:model}
	
	The starting point is the trajectory of a material point $P$ of mass
	$m$. Away from any teleportation event the motion is smooth, so $P$
	follows a curve of class $C^2$. We now make precise what it means for
	such a trajectory to contain a teleportation event. The object
	carrying a possible jump is taken as primitive, and the ordinary
	continuous trajectory is recovered as the special case in which no jump
	occurs.
	
	\begin{definition}[Trajectory with a teleportation instant]
		\label{def:traj}
		Let $I\subset\mathbb R$ be a nonempty open interval and let
		$t_0\in I$. A \emph{trajectory with teleportation instant $t_0$} is
		a map
		\[
		\gamma_P:I\setminus\{t_0\}\longrightarrow\mathbb R^3
		\]
		of class $C^2$ on $I\setminus\{t_0\}$ such that the one-sided limits
		\[
		P_{-}:=\lim_{t\to t_0^{-}}\gamma_P(t),
		\qquad
		P_{+}:=\lim_{t\to t_0^{+}}\gamma_P(t)
		\]
		exist and are finite, and such that the one-sided velocity limits
		\[
		\mathbf v_{-}:=\lim_{t\to t_0^{-}}\dot\gamma_P(t),
		\qquad
		\mathbf v_{+}:=\lim_{t\to t_0^{+}}\dot\gamma_P(t)
		\]
		exist and are finite.
	\end{definition}
	
	The instant $t_0$ is deliberately excluded from the domain: the point
	$P$ has no position at $t_0$.
	
	\begin{definition}[Teleportation]
		\label{def:teleport}
		The particle $P$ \emph{teleports} at $t_0$ if $P_{+}\neq P_{-}$. The
		vector $\boldsymbol\Delta:=P_{+}-P_{-}\in\mathbb R^3\setminus
		\{\mathbf 0\}$ is the \emph{teleportation displacement}, and
		$P_{-}$, $P_{+}$ are the \emph{departure} and \emph{arrival}
		points.
	\end{definition}
	
	\begin{remark}[Recovery of ordinary motion]
		If $P_{+}=P_{-}$ and $\mathbf v_{+}=\mathbf v_{-}$, then $\gamma_P$
		extends to a $C^{1}$ curve on all of $I$ by setting
		$\gamma_P(t_0):=P_{-}$; no teleportation occurs and ordinary
		continuous motion is recovered. Thus Definition~\ref{def:traj}
		contains ordinary motion as the special case
		$\boldsymbol\Delta=\mathbf 0$, $\Delta\mathbf v=\mathbf 0$: the
		trajectory with a possible jump is primitive, and continuity is a
		property that may or may not hold. On its domain $I\setminus\{t_0\}$
		the map $\gamma_P$ is everywhere differentiable; the question of
		differentiability \emph{at} $t_0$ does not arise, since
		$t_0\notin\operatorname{dom}\gamma_P$.
	\end{remark}
	
	\begin{remark}[The event is exogenous to the dynamics]
		\label{rem:exogenous}
		Throughout, Newton's law is imposed only on $I\setminus\{t_0\}$,
		and the event at $t_0$ is treated as an operation external to
		Newtonian dynamics. We do \emph{not} embed $\gamma_P$ as a
		distribution on all of $I$ and differentiate through $t_0$. This is
		deliberate. Were the discontinuous trajectory inserted into
		$m\,D^{2}\mathbf x=\mathbf F$ in the sense of distributions, with a
		position jump $\boldsymbol\Delta$ and a velocity jump
		$\Delta\mathbf v$ at $t_0$, one would obtain
		\[
		m\,D^{2}\mathbf x
		=m\,\mathbf a_{\mathrm{reg}}
		+m\,\Delta\mathbf v\,\delta_{t_0}
		+m\,\boldsymbol\Delta\,\delta'_{t_0},
		\]
		where $\delta_{t_0}$ is the Dirac distribution at $t_0$ and
		$\delta'_{t_0}$ its derivative. The velocity jump contributes an
		ordinary impulse $m\,\Delta\mathbf v\,\delta_{t_0}$; the position
		jump contributes the strictly more singular term
		$m\,\boldsymbol\Delta\,\delta'_{t_0}$, which nevertheless delivers
		\emph{zero} net linear impulse, since $\int_{\mathbb R}\delta'_{t_0}
		=0$. Removing $t_0$ from the domain does not by itself dispose of
		this term, because the associated distribution depends only on the
		$L^{1}_{\mathrm{loc}}$ class of $\gamma_P$. We therefore do not
		claim that the position jump is dynamically neutral within a
		distributional reading; we simply place the event outside the
		dynamics and compare the one-sided states, reading off the net
		impulse the mechanism must supply as the momentum jump
		$\mathbf p_{+}-\mathbf p_{-}$.
	\end{remark}
	
	It is essential to distinguish the two discontinuities. The
	\emph{position} jump $\boldsymbol\Delta=P_{+}-P_{-}$ is the
	teleportation itself; the \emph{velocity} jump
	$\Delta\mathbf v:=\mathbf v_{+}-\mathbf v_{-}$ may or may not vanish,
	and, as we now show, it is the velocity jump that fixes the net
	impulse.
	
	\begin{theorem}[Net impulse of the event]
		\label{thm:impulse}
		Let $P$ be a particle of constant mass $m>0$ following a trajectory
		with teleportation instant $t_0$ in the sense of
		Definition~\ref{def:traj}. Extend the momentum
		$\mathbf p=m\,\dot\gamma_P$ to a function of bounded variation on
		$I$ by assigning it its one-sided limits $\mathbf p_{\pm}=m\,
		\mathbf v_{\pm}$, and let $\mathbf F:=D\mathbf p$ be its
		distributional derivative, a finite vector-valued measure on
		compact subintervals. Then the atom of $\mathbf F$ at $t_0$ is
		\[
		\mathbf F(\{t_0\})=\mathbf p_{+}-\mathbf p_{-}=m\,\Delta\mathbf v,
		\]
		so the net linear impulse concentrated at the event is
		$\mathbf J=m\,\Delta\mathbf v$. In particular the event transfers a
		nonzero net linear impulse if and only if $\Delta\mathbf v\neq
		\mathbf 0$.
	\end{theorem}
	
	\begin{proof}
		By hypothesis the one-sided velocity limits exist and are finite, so
		$\mathbf p$ is bounded and of bounded variation on any compact
		subinterval of $I$ containing $t_0$: it is $C^{1}$ on each side of
		$t_0$ and has a single jump there. The distributional derivative of
		a $BV$ function decomposes into an absolutely continuous part, a
		jump (atomic) part and a Cantor part; here the Cantor part vanishes
		and the only atom sits at $t_0$, with mass equal to the jump
		$\mathbf p_{+}-\mathbf p_{-}=m(\mathbf v_{+}-\mathbf v_{-})
		=m\,\Delta\mathbf v$, which is nonzero precisely when
		$\Delta\mathbf v\neq\mathbf 0$ because $m>0$. Equivalently, for the
		regularized transitions $\mathbf F_\varepsilon$ obtained by
		smoothing $\mathbf p$ across $(t_0-\varepsilon,t_0+\varepsilon)$,
		\[
		\int_{t_0-\varepsilon}^{t_0+\varepsilon}\mathbf F_\varepsilon(t)\,dt
		=\mathbf p(t_0+\varepsilon)-\mathbf p(t_0-\varepsilon)
		\xrightarrow[\varepsilon\to0^{+}]{}
		\mathbf p_{+}-\mathbf p_{-}=m\,\Delta\mathbf v,
		\]
		so the impulse concentrated at $t_0$ is $m\,\Delta\mathbf v$.
	\end{proof}
	
	\begin{definition}[Momentum-compatible teleportation]
		\label{def:admissible}
		A teleportation event at $t_0$ is \emph{momentum-compatible} if the
		velocity vectors immediately before and after the event coincide,
		\[
		\Delta\mathbf v=\mathbf v_{+}-\mathbf v_{-}=\mathbf 0.
		\]
	\end{definition}
	
	\begin{corollary}[Momentum-compatibility condition]
		\label{cor:compat}
		A teleportation event transfers no net impulsive variation of linear
		momentum if and only if it is momentum-compatible. Equivalently, an
		instantaneous displacement $\boldsymbol\Delta\neq\mathbf 0$ is free
		of net impulsive momentum transfer precisely when
		$\mathbf v_{+}=\mathbf v_{-}$.
	\end{corollary}
	
	\begin{proof}
		Immediate from Theorem~\ref{thm:impulse}: the net impulse
		$\mathbf J=m\,\Delta\mathbf v$ vanishes if and only if
		$\Delta\mathbf v=\mathbf 0$.
	\end{proof}
	
	\begin{remark}[What compatibility does and does not guarantee]
		\label{rem:not-admissible}
		Corollary~\ref{cor:compat} concerns the \emph{net} linear impulse
		only. Momentum compatibility does not by itself exclude impulsive
		forces: internal impulses of opposite sign, or impulses applied to
		different regions of an extended body, may cancel in the total while
		still deforming the body. For this reason we avoid the term
		``mechanically admissible''. Full mechanical admissibility requires
		additional conditions on the spatial distribution of the forces, on
		the angular momentum, and on the resulting internal stresses; these
		are addressed in Sections~\ref{sec:extended} and~\ref{sec:deformation}.
	\end{remark}
	
	\begin{remark}[Energy is not the obstruction]
		\label{rem:energy}
		The worst case for the impulse is velocity reversal,
		$\mathbf v_{+}=-\mathbf v_{-}$ with equal speeds. There the kinetic
		energy is \emph{conserved},
		\[
		\Delta E_{\mathrm k}
		=\tfrac12 m\bigl(\|\mathbf v_{+}\|^{2}-\|\mathbf v_{-}\|^{2}\bigr)=0,
		\]
		yet the momentum jump is maximal,
		$\|\Delta\mathbf p\|=2m\|\mathbf v_{-}\|$. The mechanical violence of
		the event is therefore governed by the discontinuity of
		\emph{momentum}, not by any energy imbalance: the body's material
		responds to $\Delta\mathbf v$, which may be large even when no net
		energy is exchanged.
	\end{remark}
	
	The principal question of this work may now be stated precisely.
	
	\begin{quote}
		\emph{Under which kinematic conditions can an instantaneous
			teleportation occur without producing impulsive forces capable of
			damaging the transported body?}
	\end{quote}
	
	When the trajectory is fixed by astronomical constraints, the velocity
	vectors $\mathbf v_{-}$ and $\mathbf v_{+}$ are not free parameters.
	The next section computes $\Delta\mathbf v$ for such trajectories. We
	assume throughout that the mass is conserved during teleportation.
	
	\section{Modeling the Trajectory $\gamma_P$}
	\label{sec:curve}
	
	We build an explicit model of a point fixed to the Earth's surface,
	retaining the two dominant motions: the daily axial rotation and the
	annual orbital revolution. Higher-order contributions (precession,
	nutation, the lunar wobble, and the motion of the Solar System through
	the Galaxy) are smooth and may be added as refinements.
	
	\subsection{Kinematic decomposition}
	
	Work in the Solar-System-barycentric frame $\mathcal I$. Represent the
	orbital position of the Earth's center as $\mathbf R(t)$ and the
	position of $P$ relative to that center as $\mathbf r(t)$, so that
	\begin{equation}
		\gamma_P(t)=\mathbf R(t)+\mathbf r(t).
		\label{eq:decomp}
	\end{equation}
	Modeling the orbit as circular about the origin (the eccentricity
	$e\approx 0.0167$, and the offset of the barycenter from the Sun, are
	neglected at leading order and constitute additional approximations),
	\begin{equation}
		\mathbf R(t)=
		R_{\mathrm o}\bigl(\cos\Omega t\,\mathbf e_1
		+\sin\Omega t\,\mathbf e_2\bigr),
		\label{eq:orbit}
	\end{equation}
	where $R_{\mathrm o}$ is one astronomical unit and $\Omega$ the orbital
	angular frequency. For a point on the equator, neglecting the
	$23.44^\circ$ obliquity at leading order,
	\begin{equation}
		\mathbf r(t)=
		R_{\oplus}\bigl(\cos(\omega t+\varphi_0)\,\mathbf e_1
		+\sin(\omega t+\varphi_0)\,\mathbf e_2\bigr),
		\label{eq:spin}
	\end{equation}
	with $\omega$ the sidereal rotational frequency and $\varphi_0$ an
	initial phase. Differentiating,
	\begin{equation}
		\dot\gamma_P(t)=
		\underbrace{R_{\mathrm o}\Omega\bigl(-\sin\Omega t\,\mathbf e_1
			+\cos\Omega t\,\mathbf e_2\bigr)}_{\mathbf V_{\mathrm o}(t)}
		+\underbrace{R_{\oplus}\omega\bigl(-\sin(\omega t+\varphi_0)\,
			\mathbf e_1+\cos(\omega t+\varphi_0)\,\mathbf e_2\bigr)}_{
			\mathbf V_{\mathrm r}(t)},
		\label{eq:vel}
	\end{equation}
	with characteristic speeds $v_{\mathrm o}=R_{\mathrm o}\Omega$ and
	$v_{\mathrm r}=R_{\oplus}\omega$.
	
	\subsection{The velocity jump requires two distinct configurations}
	
	Equation~\eqref{eq:vel} defines a single $C^{\infty}$ function, for
	which $\dot\gamma_P(t_0^{-})=\dot\gamma_P(t_0^{+})=\dot\gamma_P(t_0)$
	and hence $\Delta\mathbf v=\mathbf 0$. A genuine teleportation is
	\emph{not} a one-sided limit of one smooth curve at one instant; it is
	a jump between two distinct configurations. Accordingly we model the
	trajectory by two branches sharing the instant $t_0$,
	\begin{equation}
		\gamma_{-}(t)=\mathbf R_{-}(t)+\mathbf r_{-}(t)\ (t<t_0),
		\qquad
		\gamma_{+}(t)=\mathbf R_{+}(t)+\mathbf r_{+}(t)\ (t>t_0),
		\label{eq:branches}
	\end{equation}
	whose orbital and rotational phases differ. The mechanically relevant
	quantity is
	\begin{equation}
		\Delta\mathbf v=\dot\gamma_{+}(t_0)-\dot\gamma_{-}(t_0)
		=\bigl(\mathbf V_{\mathrm o}^{+}-\mathbf V_{\mathrm o}^{-}\bigr)
		+\bigl(\mathbf V_{\mathrm r}^{+}-\mathbf V_{\mathrm r}^{-}\bigr).
		\label{eq:dv}
	\end{equation}
	
	\subsection{Two points of the same Earth at the same instant: the
		orbital velocity cancels}
	
	The magnitude of $\Delta\mathbf v$ depends decisively on the scenario.
	Consider first a teleportation between two points of the \emph{same}
	Earth at the \emph{same} instant $t_0$. Writing the velocity of a
	surface point as
	\begin{equation}
		\mathbf v=\mathbf V_{E}(t_0)+\boldsymbol\omega\times\mathbf r,
		\label{eq:surfacevel}
	\end{equation}
	where $\mathbf V_{E}(t_0)$ is the velocity of the Earth's center
	(orbital motion included) and $\mathbf r$ is the position relative to
	that center, the departure and arrival velocities share the common term
	$\mathbf V_{E}(t_0)$. Subtracting,
	\begin{equation}
		\boxed{\;\Delta\mathbf v=\boldsymbol\omega\times
			(\mathbf r_{+}-\mathbf r_{-})\;}
		\label{eq:cancel}
	\end{equation}
	so the orbital contribution cancels identically. Since
	$\|\mathbf r_{+}-\mathbf r_{-}\|\le 2R_{\oplus}$,
	\begin{equation}
		\|\Delta\mathbf v\|\le 2\,\omega R_{\oplus}
		=2\,v_{\mathrm r}\approx 9.30\times10^{2}\ \mathrm{m\,s^{-1}}.
		\label{eq:samebound}
	\end{equation}
	This is the correct kinematic bound for same-Earth, same-instant
	teleportation: only the rotational mismatch survives.
	
	\subsection{Inter-epoch and inter-frame teleportation}
	
	The much larger orbital scale enters only when departure and arrival do
	\emph{not} share the common velocity $\mathbf V_{E}(t_0)$. This happens
	in two ways: (i) the arrival occurs at a different orbital epoch (for a
	near-reversal of the orbital velocity, separated by roughly half an
	orbital period), or (ii) the arrival lies in a frame in relative
	translational motion (a different astronomical body). In these cases the
	arrival velocity is fixed by the destination state rather than by the
	Earth's instantaneous trajectory, and, in the worst case of antiparallel
	departure and arrival velocities of maximal magnitude,
	\begin{equation}
		\|\Delta\mathbf v\|\le 2\,(v_{\mathrm o}+v_{\mathrm r})
		\approx 6.05\times10^{4}\ \mathrm{m\,s^{-1}},
		\label{eq:interbound}
	\end{equation}
	an \emph{upper} bound, attained only when both one-sided speeds equal
	$v_{\mathrm o}+v_{\mathrm r}$ and point in opposite directions.
	
	\subsection{Numerical parameters}
	
	Table~\ref{tab:constants} collects the constants used. The rotational
	rate is referred to the sidereal day and the orbital speed is the mean
	value over the slightly eccentric orbit.
	
	\begin{table}[htbp]
		\centering
		\begin{tabular}{@{}llll@{}}
			\toprule
			Symbol & Quantity & Value & Source \\
			\midrule
			$R_{\oplus}$ & Earth equatorial radius
			& $6.378137\times10^{6}\ \mathrm{m}$
			& \cite{WGS84,NASAEarthFactSheet} \\
			$\omega$ & Sidereal rotation rate
			& $7.2921159\times10^{-5}\ \mathrm{rad\,s^{-1}}$
			& \cite{IERS2010,NASAEarthFactSheet} \\
			$v_{\mathrm r}$ & Equatorial rotation speed
			& $4.651\times10^{2}\ \mathrm{m\,s^{-1}}$
			& \cite{NASAEarthFactSheet} \\
			$R_{\mathrm o}$ & Orbital radius (1 AU)
			& $1.495979\times10^{11}\ \mathrm{m}$
			& \cite{IAU2012,NASAEarthFactSheet} \\
			$\Omega$ & Orbital angular rate
			& $1.99099\times10^{-7}\ \mathrm{rad\,s^{-1}}$
			& \cite{NASAEarthFactSheet} \\
			$v_{\mathrm o}$ & Mean orbital speed
			& $2.9780\times10^{4}\ \mathrm{m\,s^{-1}}$
			& \cite{NASAEarthFactSheet} \\
			$e$ & Orbital eccentricity
			& $1.67\times10^{-2}$
			& \cite{NASAEarthFactSheet} \\
			\bottomrule
		\end{tabular}
		\caption{Astronomical constants used to model the trajectory
			$\gamma_P$.}
		\label{tab:constants}
	\end{table}
	
	With these values the same-Earth bound~\eqref{eq:samebound} is
	$\|\Delta\mathbf v\|\le 9.30\times10^{2}\ \mathrm{m\,s^{-1}}$, whereas
	the inter-frame bound~\eqref{eq:interbound} is
	$\|\Delta\mathbf v\|\le 6.05\times10^{4}\ \mathrm{m\,s^{-1}}$. These two
	figures set the scale against which material strength is compared in
	Section~\ref{sec:deformation}. The corresponding maximal impulse per
	unit mass is $\|\mathbf J\|/m=\|\Delta\mathbf v\|$.
	
	\section{Teleportation of an Extended Body}
	\label{sec:extended}
	
	We generalize to an extended rigid body, viewed as a continuum of
	material points. Let $\mathcal B\subset\mathbb R^3$ be the reference
	configuration, with $X_0$ a fixed reference point, and let each material
	point $X\in\mathcal B$ carry a trajectory in the sense of
	Definition~\ref{def:traj}, all sharing the instant $t_0$.
	
	\subsection{Rigid velocity field, orientation included}
	
	The configuration of a rigid body just before and just after the event
	is
	\begin{equation}
		\mathbf y_{\pm}(X)=\mathbf c_{\pm}+Q_{\pm}\,(X-X_0),
		\qquad Q_{\pm}\in SO(3),
		\label{eq:rigidconfig}
	\end{equation}
	where $\mathbf c_{\pm}$ are the center-of-mass positions and $Q_{\pm}$
	the orientations. The one-sided velocity fields are then
	\begin{equation}
		\mathbf v_{\pm}(X)=\mathbf V_{\pm}
		+\boldsymbol\omega_{\pm}\times Q_{\pm}(X-X_0),
		\label{eq:rigidvel}
	\end{equation}
	with $\mathbf V_{\pm}$ the center-of-mass velocities and
	$\boldsymbol\omega_{\pm}$ the angular velocities. Hence the pointwise
	velocity jump is
	\begin{equation}
		\Delta\mathbf v(X)=\Delta\mathbf V
		+\boldsymbol\omega_{+}\times Q_{+}(X-X_0)
		-\boldsymbol\omega_{-}\times Q_{-}(X-X_0).
		\label{eq:rigidjumpfull}
	\end{equation}
	Only when the orientation is preserved, $Q_{+}=Q_{-}=:Q$, does this
	collapse to the affine form
	\begin{equation}
		\Delta\mathbf v(X)=\Delta\mathbf V
		+\Delta\boldsymbol\omega\times Q(X-X_0),
		\label{eq:rigidjump}
	\end{equation}
	which is the expression used in what follows. The impulse density the
	body must exchange is $\mathbf j(X)=\rho(X)\,\Delta\mathbf v(X)$, with
	$\rho$ the mass density.
	
	\begin{proposition}[Kinematic compatibility of a rigid event]
		\label{prop:rigid}
		Assume the orientation is preserved, $Q_{+}=Q_{-}$, and that
		$\mathcal B$ is not contained in a line through its center of mass.
		Then the pointwise velocity jump vanishes for every material point,
		$\Delta\mathbf v(X)=\mathbf 0$ for all $X\in\mathcal B$, if and only
		if both the translational and the angular velocities are preserved,
		\[
		\Delta\mathbf V=\mathbf 0
		\qquad\text{and}\qquad
		\Delta\boldsymbol\omega=\mathbf 0.
		\]
	\end{proposition}
	
	\begin{proof}
		If both jumps vanish, \eqref{eq:rigidjump} gives
		$\Delta\mathbf v(X)=\mathbf 0$ for every $X$. Conversely, if
		$\Delta\mathbf v(X)=\mathbf 0$ for all $X$, evaluating at the center
		of mass gives $\Delta\mathbf V=\mathbf 0$; substituting back forces
		$\Delta\boldsymbol\omega\times Q(X-X_0)=\mathbf 0$ for all $X$, and
		since $Q$ is invertible and $\mathcal B$ is not contained in a line
		through the center of mass, the vectors $Q(X-X_0)$ span a set not
		confined to a line, whence $\Delta\boldsymbol\omega=\mathbf 0$.
	\end{proof}
	
	Proposition~\ref{prop:rigid} is a purely \emph{kinematic} statement
	about the pointwise momentum of each material element. It does not, by
	itself, assert the presence or absence of internal stress; that
	requires the continuum balance discussed next.
	
	\subsection{Net impulse, body force and internal stress are distinct}
	\label{subsec:distinct}
	
	It is tempting to conclude that a nonzero $\Delta\mathbf v(X)$ must be
	transmitted through the body by internal forces and therefore produces
	internal stress. This does not follow from kinematics alone. In
	continuum mechanics the local balance of linear momentum reads
	\begin{equation}
		\rho\,\dot{\mathbf v}=\operatorname{div}\boldsymbol\sigma
		+\rho\,\mathbf b,
		\label{eq:cauchy}
	\end{equation}
	where $\boldsymbol\sigma$ is the Cauchy stress tensor and $\mathbf b$
	the body force per unit mass. A prescribed velocity change $\Delta
	\mathbf v(X)$ may be produced by
	\begin{enumerate}[label=(\alph*)]
		\item surface tractions transmitted between regions of the body
		(elastic waves), so that the term $\operatorname{div}
		\boldsymbol\sigma$ dominates; or
		\item a body force $\mathbf b$ acting directly on each material
		element; or
		\item a combination of the two.
	\end{enumerate}
	If the mechanism imparts a spatially \emph{uniform} translational jump,
	$\Delta\mathbf v(X)=\mathbf V_0$, through a uniform body force per unit
	mass, there is no velocity gradient and hence no internal stress: the
	body simply accelerates as a whole, exactly as it does under uniform
	gravity in free fall. Internal stress arises when the momentum is
	delivered non-uniformly---for instance at a surface or an end, whence it
	propagates as a stress wave. Consequently the internal stress is not
	determined by $\Delta\mathbf v$ alone but by the spatial and temporal
	distribution of the loading.
	
	Two consequences follow. First, the safe reading of
	Proposition~\ref{prop:rigid} is that $\Delta\mathbf v(X)=\mathbf 0$ for
	all $X$ guarantees no pointwise change of momentum of any element---a
	necessary condition for a stress-free event, not a computation of
	stress. Second, an angular mismatch is qualitatively worse than a
	translational one: by \eqref{eq:rigidjump} the field
	$\Delta\boldsymbol\omega\times Q(X-X_0)$ is intrinsically
	non-uniform, growing linearly with distance from the center of mass, so
	it cannot be supplied by a uniform body force. Cancelling it would
	require a body force finely tuned to each element---in effect, control
	of the momentum of every material point---whereas any mechanism that
	grips the body over part of its extent delivers the momentum locally and
	generates the stress waves estimated below.
	
	\section{Deformation of Bodies}
	\label{sec:deformation}
	
	Under localized delivery, the impulse density must be transmitted by
	internal forces. This section recalls the elements of linear elasticity
	needed for an order-of-magnitude estimate, and states carefully the
	range in which that estimate is meaningful
	\cite{LandauLifshitzElasticity,Gere2012,TimoshenkoGoodier1970}.
	
	\subsection{Stress, strain and Young's modulus}
	
	For a prismatic bar of length $L_0$ and cross-section $A$ under an axial
	force $F$, the normal stress and strain are
	\begin{equation}
		\sigma=\frac{F}{A},
		\qquad
		\varepsilon=\frac{\Delta L}{L_0},
		\label{eq:stressstrain}
	\end{equation}
	related within the linear elastic regime by Hooke's law
	\begin{equation}
		\sigma=E\,\varepsilon,
		\label{eq:hooke}
	\end{equation}
	where $E$ is Young's modulus (for structural steel $E\approx 200\
	\mathrm{GPa}$). Hooke's law holds only up to the yield stress
	$\sigma_Y$; fracture occurs at the ultimate tensile strength
	$\sigma_U$. Table~\ref{tab:materials} lists representative values.
	
	\begin{table}[htbp]
		\centering
		\begin{tabular}{@{}lccc@{}}
			\toprule
			Material & $E$ (GPa) & $\sigma_Y$ (MPa) & $\sigma_U$ (MPa) \\
			\midrule
			Structural steel (mild) & $200$ & $250$ & $400$ \\
			Stainless steel 304 & $193$ & $215$ & $505$ \\
			Aluminium alloy (6061-T6) & $69$ & $276$ & $310$ \\
			Titanium alloy (Ti-6Al-4V) & $114$ & $880$ & $950$ \\
			\bottomrule
		\end{tabular}
		\caption{Representative room-temperature, quasi-static mechanical
			properties \cite{AmeswebSteel,ASM2000,MatWeb}.}
		\label{tab:materials}
	\end{table}
	
	\subsection{Acoustic estimate and its range of validity}
	
	A one-dimensional momentum-transfer argument gives a first estimate of
	the stress required to communicate a velocity jump $\Delta v=\|\Delta
	\mathbf v\|$ across the body by an elastic wave travelling at the bar
	speed $c=\sqrt{E/\rho}$. The associated acoustic (Hugoniot) stress is
	\begin{equation}
		\sigma_{\mathrm{imp}}=\rho\,c\,\Delta v=\sqrt{\rho E}\;\Delta v,
		\label{eq:acoustic}
	\end{equation}
	the product of the acoustic impedance $\rho c$ and the velocity jump.
	This relation is not a universal law converting any velocity jump into a
	stress; it is the linear response of a specific one-dimensional model
	under localized loading, and it presupposes small strains. Its own
	output reveals the limitation: combining \eqref{eq:hooke},
	\eqref{eq:acoustic} and $E=\rho c^2$ gives the implied strain
	\begin{equation}
		\varepsilon=\frac{\sigma_{\mathrm{imp}}}{E}
		=\frac{\rho c\,\Delta v}{\rho c^{2}}=\frac{\Delta v}{c}.
		\label{eq:strain}
	\end{equation}
	Thus the estimate is self-consistent only while $\Delta v\ll c$.
	
	\paragraph{Same-Earth (rotational) case.}
	For steel, $\rho\approx 7.85\times10^{3}\ \mathrm{kg\,m^{-3}}$,
	$E\approx 200\ \mathrm{GPa}$, so $c\approx 5.05\times10^{3}\
	\mathrm{m\,s^{-1}}$. With the same-Earth bound $\Delta v\approx
	9.30\times10^{2}\ \mathrm{m\,s^{-1}}$ from~\eqref{eq:samebound},
	\[
	\sigma_{\mathrm{imp}}
	\approx (7.85\times10^{3})(5.05\times10^{3})(9.30\times10^{2})
	\approx 3.7\times10^{10}\ \mathrm{Pa}
	=3.7\times10^{4}\ \mathrm{MPa},
	\]
	about $90$ times the ultimate strength of mild steel and roughly two
	orders of magnitude above it. The implied strain is
	$\varepsilon=\Delta v/c\approx 0.18$: large enough to place the response
	well beyond the small-strain, linear-elastic range (into gross
	plasticity and fracture for ordinary metals), yet not physically
	absurd. The estimate should therefore be read as a strong indication
	that any \emph{localized} delivery of this rotational mismatch drives
	common structural metals past failure, while the highest-strength
	materials (e.g.\ titanium alloys, $\sigma_U\approx 950\ \mathrm{MPa}$)
	remain exceeded by a factor of order tens.
	
	\paragraph{Inter-frame case.}
	With the inter-frame bound $\Delta v\approx 6.05\times10^{4}\
	\mathrm{m\,s^{-1}}$ from~\eqref{eq:interbound},
	\[
	\sigma_{\mathrm{imp}}
	\approx 2.4\times10^{12}\ \mathrm{Pa}=2.4\times10^{6}\ \mathrm{MPa},
	\]
	which exceeds the ultimate strength of steel
	($\sigma_U\approx 400\ \mathrm{MPa}$) by a factor of about $6000$, i.e.\
	by $\log_{10}(6000)\approx 3.8$ orders of magnitude---more than three,
	nearly four. Here, however, the implied strain
	$\varepsilon=\Delta v/c\approx 12$ (about $1200\%$) is entirely
	incompatible with linear elasticity, so the figure $2.4\times10^{12}\
	\mathrm{Pa}$ must \emph{not} be presented as a quantitative prediction
	of the real stress. It indicates only that the response would leave the
	linear-elastic regime long before reaching such a state.
	
	\subsection{Interpretation}
	
	These estimates indicate that, under localized delivery of momentum and
	unless the teleportation mechanism enforces the kinematic compatibility
	of Proposition~\ref{prop:rigid} to high precision, the induced stresses
	exceed the quasi-static strengths of the representative materials of
	Table~\ref{tab:materials}. Three caveats accompany this conclusion.
	First, the outcome is conditional on \emph{how} the momentum is
	delivered (Subsection~\ref{subsec:distinct}); a perfectly uniform body
	force would produce no internal stress. Second, a dynamic, extreme-rate
	loading should not be compared naively with quasi-static strengths
	without accounting for strain-rate dependence, multiaxial stress state,
	temperature and dynamic material properties. Third, at the strains
	implied here a faithful treatment requires plasticity, shock formation,
	rate dependence, heating, fracture and possibly phase change. A precise
	elastodynamic and shock analysis, together with numerical simulation of
	the wave field generated by the impulse density, is left for future
	work.
	
	\section{Conclusion}
	\label{sec:conclusion}
	
	We have proposed a classical-mechanical framework for the mechanical
	analysis of instantaneous teleportation. Modeling the trajectory of each
	material point as a $C^2$ curve on a time domain punctured at the
	teleportation instant, and treating the event as exogenous to Newtonian
	dynamics (Remark~\ref{rem:exogenous}), we represented teleportation as a
	jump discontinuity of position at an instant when the body is absent
	from space. The mathematically defensible core is the following: if the
	kinematic states immediately before and after the event carry different
	linear momenta, the mechanism must supply a net linear impulse equal to
	their difference,
	\[
	\mathbf J=\mathbf p_{+}-\mathbf p_{-}=m\,(\mathbf v_{+}-\mathbf v_{-}),
	\]
	carried by the velocity discontinuity; the position jump contributes no
	net linear impulse (Theorem~\ref{thm:impulse},
	Corollary~\ref{cor:compat}). For rigid bodies with preserved
	orientation the compatibility condition splits into the preservation of
	both translational and angular velocity
	(Proposition~\ref{prop:rigid}), understood as a kinematic statement
	rather than a claim about internal stress.
	
	Applied to a point on the Earth, the framework yields a kinematic fact
	that corrects a natural but mistaken estimate: for two points of the
	same Earth at the same instant the orbital velocity cancels, so the
	relevant velocity jump is bounded by
	$2\omega R_{\oplus}\approx 9.3\times10^{2}\ \mathrm{m\,s^{-1}}$, not by
	the orbital scale $6\times10^{4}\ \mathrm{m\,s^{-1}}$, which pertains
	only to inter-epoch or inter-frame teleportation. Even the corrected
	rotational jump, under localized delivery, produces acoustic stresses
	tens to a hundred times the strength of common structural metals; but
	the induced stress depends on how the mechanism transfers momentum to
	the material, and the implied strains lie outside the range of linear
	elasticity, so the numbers are indicative rather than quantitative.
	
	The unavoidable qualitative conclusion, within classical mechanics, is
	that a hypothetical teleportation device would have to enforce kinematic
	compatibility with the destination's state of motion---and distribute
	the residual momentum transfer with great uniformity---to avoid damaging
	the transported body. Several directions remain open: a full
	elastodynamic and shock simulation of the impulsive field, the
	incorporation of orbital eccentricity, barycentric offset and axial
	obliquity, the treatment of deformable rather than rigid bodies, and a
	thermodynamic account of the information lost at the jump, the last
	deliberately left outside the mechanical framework of this paper.


\begin{thebibliography}{99}
		
		\bibitem{Bennett1993}
		C.H. Bennett, G. Brassard, C. Crépeau, R. Jozsa, A. Peres and
		W.K. Wootters, Phys. Rev. Lett. \textbf{70}, 1895 (1993).
		
		\bibitem{Bouwmeester1997}
		D. Bouwmeester, J.-W. Pan, K. Mattle, M. Eibl, H. Weinfurter and
		A. Zeilinger, Nature \textbf{390}, 575 (1997).
		
		\bibitem{NielsenChuang2010}
		M.A. Nielsen and I.L. Chuang, \textit{Quantum Computation and
			Quantum Information} (Cambridge University Press, Cambridge,
		2010), 10th Anniversary ed.
		
		\bibitem{Goldstein2002}
		H. Goldstein, C. Poole and J. Safko, \textit{Classical Mechanics}
		(Addison-Wesley, San Francisco, 2002), 3rd ed.
		
		\bibitem{LandauLifshitz1976}
		L.D. Landau and E.M. Lifshitz, \textit{Mechanics}
		(Butterworth-Heinemann, Oxford, 1976), 3rd ed., v. 1.
		
		\bibitem{LandauLifshitzElasticity}
		L.D. Landau and E.M. Lifshitz, \textit{Theory of Elasticity}
		(Butterworth-Heinemann, Oxford, 1986), 3rd ed., v. 7.
		
		\bibitem{Gere2012}
		J.M. Gere and B.J. Goodno, \textit{Mechanics of Materials}
		(Cengage Learning, Stamford, 2012), 8th ed.
		
		\bibitem{TimoshenkoGoodier1970}
		S.P. Timoshenko and J.N. Goodier, \textit{Theory of Elasticity}
		(McGraw-Hill, New York, 1970), 3rd ed.
		
		\bibitem{WGS84}
		National Imagery and Mapping Agency, \textit{Department of Defense
			World Geodetic System 1984 (WGS 84)}, NIMA TR8350.2 (National
		Imagery and Mapping Agency, Bethesda, 2000), 3rd ed.
		
		\bibitem{IERS2010}
		G. Petit and B. Luzum (eds.), \textit{IERS Conventions (2010)},
		IERS Technical Note No. 36 (Verlag des Bundesamts für
		Kartographie und Geodäsie, Frankfurt am Main, 2010).
		
		\bibitem{IAU2012}
		International Astronomical Union, \textit{Resolution B2 on the
			Re-definition of the Astronomical Unit of Length}, XXVIII General
		Assembly (International Astronomical Union, Beijing, 2012).
		
		\bibitem{NASAEarthFactSheet}
		D.R. Williams, \textit{NASA Earth Fact Sheet} (NASA Goddard Space
		Flight Center, National Space Science Data Center), available at
		\url{https://nssdc.gsfc.nasa.gov/planetary/factsheet/earthfact.html},
		accessed on Aug. 2026.
		
		\bibitem{AmeswebSteel}
		AMESWeb, \textit{Young's Modulus of Steel}, available at
		\url{https://amesweb.info/Materials/Youngs-Modulus-of-Steel.aspx},
		accessed on Aug. 2026.
		
		\bibitem{ASM2000}
		ASM International, \textit{ASM Handbook, Volume 1: Properties and
			Selection: Irons, Steels, and High-Performance Alloys} (ASM
		International, Materials Park, 2000).
		
		\bibitem{MatWeb}
		MatWeb LLC, \textit{Material Property Data}, available at
		\url{https://www.matweb.com}, accessed on Aug. 2026.
		
	\end{thebibliography}
\end{document}